\documentclass[a4paper]{amsart}
\usepackage{graphicx}
\usepackage{amssymb}
\usepackage{amsmath}
\usepackage{amsthm}
\usepackage{amscd}
\usepackage[all,2cell]{xy}

\UseAllTwocells \SilentMatrices
\newtheorem{thm}{Theorem}[section]

\newtheorem{lem}[thm]{Lemma}

\newtheorem{prop}[thm]{Proposition}
\theoremstyle{definition}

\theoremstyle{remark}

\numberwithin{equation}{section}

\begin{document}
\title[A note on standard equivalences]
{A note on standard equivalences}

\author[Xiao-Wu Chen] {Xiao-Wu Chen}

\date{\today}

\thanks{E-mail:
xwchen$\symbol{64}$mail.ustc.edu.cn}

\maketitle

\dedicatory{}%
\commby{}%

\begin{abstract}
We prove that any derived equivalence between triangular algebras is standard, that is, it is isomorphic to the derived tensor functor given by a two-sided tilting complex.
\end{abstract}

\section{Introduction}

Let $k$ be a field. We will require that all categories and functors are $k$-linear. Let  $A$ be a finite dimensional $k$-algebra. We denote by $A\mbox{-mod}$ the category of finite dimensional left $A$-modules and by $\mathbf{D}^b(A\mbox{-mod})$ its bounded derived category.

Let $B$ be another finite dimensional $k$-algebra. We will require that $k$ acts centrally on any $B$-$A$-bimodule. Recall that a \emph{two-sided tilting complex} is a bounded complex $X$ of $B$-$A$-bimodules such that the derived tensor functor gives an equivalence $X\otimes_A^\mathbf{L}-\colon \mathbf{D}^b(A\mbox{-mod})\rightarrow \mathbf{D}^b(B\mbox{-mod})$.

A triangle equivalence $F\colon \mathbf{D}^b(A\mbox{-mod})\rightarrow \mathbf{D}^b(B\mbox{-mod})$ is said to be \emph{standard} if it is isomorphic, as a triangle functor, to $X\otimes_A^\mathbf{L}-$ for some two-sided tilting complex $X$. It is an open question whether all triangle equivalences are standard. The aim of this note is to answer this question affirmatively in a special case.

Recall that the algebra $A$ is \emph{triangular} provided that the Ext-quiver of $A$ has no oriented cycles.  There are explicit examples of algebras $A$ and $B$, which are derived equivalent such that $A$ is triangular, but $B$ is not; consult the top of  \cite[p.21]{BGS}. It makes sense to have the following notion: an algebra $A$ is \emph{derived-triangular} if it is derived equivalent to a triangular algebra.

\begin{thm} \label{thm:main}
Let $A$ be a derived-triangular algebra. Then any triangle equivalence $ F\colon \mathbf{D}^b(A\mbox{-{\rm mod}})\rightarrow \mathbf{D}^b(B\mbox{-{\rm mod}})$ is standard.
\end{thm}

We observe that a derived-triangular algebra has finite global dimension. The converse is not true in general. Indeed, let $A$ be a non-triangular algebra with two simple modules which has finite global dimension. Then $A$ is not derived-triangular. Indeed, any triangular algebra $B$ which is derived equivalent to $A$ has two simple modules and thus is hereditary. This forces that the algebra $A$ is triangular, yielding a contradiction.

We recall that a piecewise hereditary algebra is triangular. In particular, Theorem \ref{thm:main} implies that the assumption on the standardness of the auto-equivalence in \cite[Section 4]{Kel} is superfluous.

The proof of Theorem \ref{thm:main} is a rather immediate application of \cite[Theorem 4.7]{AR}, which characterizes certain triangle functors between the bounded homotopy categories of Orlov categories. Here, we observe that the category of projective modules over a triangular algebra is naturally an Orlov category.

We refer to \cite{Rin, Zim} for unexplained notions in the representation theory of algebras.

\section{The bounded homotopy category of an Orlov category}

Let $\mathcal{A}$ be a $k$-linear additive category, which is Hom-finite and has split idempotents. Here, the Hom-finitess means that all the Hom spaces are finite dimensional. It follows that $\mathcal{A}$ is a  Krull-Schmidt category.

We denote by ${\rm Ind}\;\mathcal{A}$ a complete set of representatives of indecomposable objects in $\mathcal{A}$. The category $\mathcal{A}$ is called \emph{bricky} if the endomorphism algebra of  each indecomposable object is a division algebra.

 We slightly generalize \cite[Definition 4.1]{AR}.  A bricky category $\mathcal{A}$ is called an \emph{Orlov category} provided that there is a degree function ${\rm deg}\colon {\rm Ind}\; \mathcal{A}\rightarrow \mathbb{Z}$ with the following property: for any indecomposable objects $P, P'$ having ${\rm Hom}_\mathcal{A}(P, P')\neq 0$, we have that $P\simeq P'$ or ${\rm deg}(P)>{\rm deg}(P')$. An object $X$ in $\mathcal{A}$ is homogeneous of degree $n$, if it is isomorphic to a finite direct sum of indecomposables of degree $n$. An additive functor $F\colon \mathcal{A}\rightarrow \mathcal{A}$ is \emph{homogeneous} if it sends homogeneous objects to homogenous objects and preserves their degrees.

Let $A$ be a finite dimensional $k$-algebra. We denote by $\{S_1, S_2, \cdots, S_n\}$ a complete set of representatives of simple $A$-modules. Denote by $P_i$ the projective cover of $S_i$. We recall that the \emph{Ext-quiver} $Q_A$ of $A$ is defined as follows. The vertex set of $Q_A$ equals $\{1, 2, \cdots, n\}$, and there is a unique arrow from $i$ to $j$ provided that ${\rm Ext}^1_A(S_i, S_j)\neq 0$. The algebra $A$ is \emph{triangular} provided that $Q_A$ has no oriented cycles.

Let $A$ be a triangular algebra. We denote by $Q_A^0$ the set of sources in $Q_A$. Here, a vertex is a source if there is no arrow ending at it. For each $d\geq 1$, we define the set $Q_A^d$ inductively,  such that a vertex $i$ belongs to $Q_A^d$ if and only if any arrow ending at $i$ necessarily starts at $\bigcup_{0\leq m\leq d-1} Q_A^{m}$. It follows that $Q_A^0\subseteq Q_A^1\subseteq Q_A^2\subseteq \cdots$ and that $\bigcup_{d\geq 0}Q_A^d=\{1, 2, \cdots, n\}$.  We mention that this construction can be found in \cite[p.42]{Rin}.

We denote by $A\mbox{-proj}$ the category of finite dimensional projective $A$-modules. Then $\{P_1, P_2, \cdots, P_n\}$ is a complete set of representatives of indecomposables in $A\mbox{-proj}$. For each $1\leq i\leq n$, we define ${\rm deg}(P_i)=d$ such that $i\in Q_A^d$ and $i\notin Q_A^{d-1}$.

The following example of an Orlov category seems to be well known.

\begin{lem}\label{lem:exm}
Let $A$ be a triangular algebra. Then $A\mbox{-{\rm proj}}$ is an Orlov category with the above degree function. Moreover, any equivalence $F\colon A\mbox{-{\rm proj}}\rightarrow A\mbox{-{\rm proj}}$ is homogeneous.
\end{lem}

\begin{proof}
Since $A$ is triangular, it is well known that ${\rm End}_A(P_i)$ is isomorphic to ${\rm End}_A(S_i)$, which is a division algebra. Then $A\mbox{-{\rm proj}}$ is bricky. We recall that for $i\neq j$ with ${\rm Hom}_A(P_i, P_j)\neq 0$, there is a path from $j$ to $i$ in $Q_A$. From the very construction, we infer that for an arrow $\alpha\colon a\rightarrow b$ with $b\in Q_A^{d}$, we have $a\in Q_A^{d-1}$. Then we are done by the following consequence: if there is a path from $j$ to $i$ in $Q_A$, then ${\rm deg}(P_j)<{\rm deg}(P_i)$.

For the final statement, we observe that the equivalence $F$ extends to an auto-equivalence on $A\mbox{-mod}$, and thus induces an automorphism of $Q_A$. The automorphism preserves the subsets $Q_A^d$. Consequently, the equivalence $F$ preserves degrees, and is homogeneous.
\end{proof}

Let $\mathcal{A}$ be a $k$-linear additive category as above. We denote by $\mathbf{K}^b(\mathcal{A})$ the homotopy category of bounded complexes in $\mathcal{A}$. Here, a complex $X$ is visualized as $\cdots \rightarrow X^{n-1}\stackrel{d^{n-1}_X}\rightarrow X^n \stackrel{d^n_X}\rightarrow X^{n+1}\rightarrow \cdots$, where the differentials satisfy $d^n_X\circ d_X^{n-1}=0$. The translation functor on $\mathbf{K}^b(\mathcal{A})$ is denoted by $[1]$, whose $n$-th power is denoted by $[n]$.

We view an object $A$ in $\mathcal{A}$ as a stalk complex concentrated at degree zero, which is still denoted by $A$. In this way, we identify $\mathcal{A}$ as a full subcategory of $\mathbf{K}^b(\mathcal{A})$.

We are interested in triangle functors on $\mathbf{K}^b(\mathcal{A})$. We recall that a \emph{triangle functor} $(F, \theta)$ consists of an additive functor $F\colon \mathbf{K}^b(\mathcal{A})\rightarrow \mathbf{K}^b(\mathcal{A})$ and a natural isomorphism $\theta\colon [1]F\rightarrow F[1]$, which preserves triangles. More precisely, for any triangle $X\rightarrow Y\rightarrow Z\stackrel{h}\rightarrow X[1]$ in $\mathbf{K}^b(\mathcal{A})$, the sequence $FX\rightarrow FY\rightarrow FZ \stackrel{\theta_X\circ F(h)}\longrightarrow (FX)[1]$ is a triangle. We refer to $\theta$ as the \emph{commutating isomorphism} for $F$. A natural transformation between triangle functors is required to respect the two commutating isomorphisms.

For a triangle functor $(F, \theta)$, the commutating isomorphism $\theta$ is \emph{trivial} if $[1]F=F[1]$ and $\theta$ is the identity transformation. In this case, we suppress $\theta$ and write $F$ for the triangle functor.

Any additive functor $F\colon \mathcal{A}\rightarrow \mathcal{A}$ gives rise to a triangle functor $\mathbf{K}^b(F)\colon \mathbf{K}^b(\mathcal{A})\rightarrow \mathbf{K}^b(\mathcal{A})$, which acts on complexes componentwise. The commuting isomorphism for $\mathbf{K}^b(F)$ is trivial. Similarly, any natural transformation $\eta\colon F\rightarrow F'$ extends to a natural transformation $\mathbf{K}^b(\eta)\colon \mathbf{K}^b(F)\rightarrow \mathbf{K}^b(F')$ between triangle functors.

The following fundamental result is due to \cite[Theorem 4.7]{AR}.

\begin{prop}\label{prop:AR}
Let $\mathcal{A}$ be an Orlov category, and let $(F, \theta)\colon \mathbf{K}^b(\mathcal{A})\rightarrow \mathbf{K}^b(\mathcal{A})$ be a triangle functor such that $F(\mathcal{A})\subseteq \mathcal{A}$. We assume further that $F|_\mathcal{A}\colon \mathcal{A}\rightarrow \mathcal{A}$ is homogeneous. Let $F_1, F_2\colon \mathcal{A}\rightarrow \mathcal{A}$ be two homogeneous functors.
\begin{enumerate}
\item There is a unique natural isomorphism $(F, \theta)\rightarrow \mathbf{K}^b(F|_\mathcal{A})$ of triangle functors, which is the identity on the full subcategory $\mathcal{A}$.
    \item Any natural transformation $\mathbf{K}^b(F_1)\rightarrow \mathbf{K}^b(F_2)$ of triangle functors is of the form $\mathbf{K}^b(\eta)$ for a unique natural transformation $\eta\colon F_1\rightarrow F_2$.
\end{enumerate}
\end{prop}

\begin{proof}
The existence of the natural isomorphism in (1) is due to \cite[Theorem 4.7]{AR}; compare \cite[Remark 4.8]{AR}. The uniqueness follows from the commutative diagram (4.10) and Lemma 4.5(2) in \cite{AR}, by induction on the support of a complex in the sense of \cite[Subsection 4.1]{AR}. Here, we emphasize that the commutating isomorphism $\theta$ is used in the construction of the natural isomorphism on stalk complexes; compare the second paragraph in \cite[p.1541]{AR}.

The statement (2) follows from the uniqueness part of (1).
\end{proof}

Recall that $\mathbf{D}^b(A\mbox{-{\rm mod}})$ denotes the bounded derived category of $A\mbox{-mod}$. We identify $A\mbox{-mod}$ as the full subcategory of $\mathbf{D}^b(A\mbox{-{\rm mod}})$ formed by stalk complexes concentrated at degree zero. We denote by $H^n(X)$ the $n$-th cohomology of a complex $X$.

The following observation is immediate.

\begin{lem}
Let $A$ be a finite dimensional algebra, and let $F\colon \mathbf{D}^b(A\mbox{-{\rm mod}})\rightarrow \mathbf{D}^b(A\mbox{-{\rm mod}})$ be a triangle equivalence with $F(A)\simeq A$. Then we have $F(A\mbox{-{\rm mod}})=A\mbox{-{\rm mod}}$, and thus the restricted  equivalence $F|_{A\mbox{-}{\rm mod}}\colon A\mbox{-{\rm
mod}}\rightarrow A\mbox{-{\rm mod}}$.
\end{lem}

\begin{proof}
We use the canonical isomorphisms $H^n(X)\simeq {\rm Hom}_{\mathbf{D}^b(A\mbox{-}{\rm mod})}(A[-n], X)$. It follows that both $F$ and its quasi-inverse send stalk complexes to stalk complexes. Then we are done.
\end{proof}

We assume that we are given an equivalence $F\colon A\mbox{-{\rm mod}}\rightarrow A\mbox{-{\rm mod}}$ with $F(A)\simeq A$. Then there is an algebra automorphism $\sigma\colon A\rightarrow A$ such that $F$ is isomorphic to $_\sigma A_1\otimes_A-$. Here, the $A$-bimodule $_\sigma A_1$ is given  by the regular right $A$-module, where the left $A$-module is twisted by $\sigma$. This bimodule is invertible and thus viewed as a two-sided tilting complex. We refer to \cite[Section 6.5]{Zim} for details on two-sided tilting complexes and standard equivalences.

We now combine the above results.

\begin{prop}\label{prop:part}
Let $A$ be a triangular algebra, and let  $(F, \theta)\colon \mathbf{D}^b(A\mbox{-{\rm mod}})\rightarrow \mathbf{D}^b(A\mbox{-{\rm mod}})$ be a triangle equivalence with  $F(A)\simeq A$. We recall the algebra automorphism $\sigma$ given by the restricted equivalence $F|_{A\mbox{-}{\rm mod}}$, and the $A$-bimodule $_\sigma A_1$. Then there is a natural isomorphism $(F, \theta)\rightarrow {_\sigma A_1}\otimes^\mathbf{L}_A-$ of triangle functors. In particular, the triangle equivalence $(F, \theta)$ is standard.
\end{prop}

\begin{proof}
Since the algebra $A$ is triangular, it has finite global dimension. The natural functor $\mathbf{K}^b(A\mbox{-proj})\rightarrow \mathbf{D}^b(A\mbox{-{\rm mod}})$ is a triangle equivalence. We identify these two categories. Therefore, the triangle functor $(F, \theta)\colon \mathbf{K}^b(A\mbox{-proj})\rightarrow \mathbf{K}^b(A\mbox{-proj})$ restricts to an equivalence $F|_{A\mbox{-}{\rm proj}}$, which is isomorphic to ${_\sigma A_1}\otimes_A-$. By Lemma \ref{lem:exm}, the statements in Proposition \ref{prop:AR} apply in our situation. Consequently, we have an isomorphism between $(F, \theta)$ and $\mathbf{K}^b({_\sigma A_1}\otimes_A-)$. Then we are done.
\end{proof}

\section{The proof of Theorem \ref{thm:main}}

We now prove Theorem \ref{thm:main}. In what follows, for simplicity, when writing a triangle functor we suppress its commutating isomorphism.

We first assume that the algebra $A$ is triangular. The complex $F(A)$ is a one-sided tilting complex. By \cite[Theorem 6.4.1]{Zim}, there is a two-sided tilting complex $X$ of $B$-$A$-bimodules with an isomorphism $X\rightarrow F(A)$ in $\mathbf{D}^b(B\mbox{-{\rm mod}})$. Denote by $G$ a quasi-inverse of the standard equivalence $X\otimes^\mathbf{L}_A-\colon \mathbf{D}^b(A\mbox{-{\rm mod}})\rightarrow \mathbf{D}^b(B\mbox{-{\rm mod}})$. Then the triangle functor $GF\colon \mathbf{D}^b(A\mbox{-{\rm mod}})\rightarrow \mathbf{D}^b(A\mbox{-{\rm mod}})$ satisfies $GF(A)\simeq A$. Proposition \ref{prop:part} implies that $GF$ is standard, and thus $F$ is isomorphic to the composition of $X\otimes^\mathbf{L}_A-$ and a standard equivalence. Then we are done in this case by the well-known fact that the composition of two standard equivalences is standard.

In general, let $A$ be derived-triangular. Assume that $A'$ is a triangular algebra which is derived equivalent to $A$. By \cite[Proposition 6.5.5]{Zim}, there is a standard equivalence $F'\colon \mathbf{D}^b(A'\mbox{-{\rm mod}})\rightarrow \mathbf{D}^b(A\mbox{-{\rm mod}})$. The above argument implies that the composition $FF'$ is standard. Recall from \cite[Proposition 6.5.6]{Zim} that a quasi-inverse $F'^{-1}$ of $F'$ is standard. We are done by observing that $F$ is isomorphic to the composition $(FF')F'^{-1}$, a composition of two standard equivalences.

\vskip 10pt

\noindent {\bf Acknowledgements}\quad  We thank Martin Kalck for pointing out the example in \cite{BGS}. The author is supported by National Natural Science Foundation of China (No. 11201446), NCET-12-0507, and the Fundamental Research Funds for the Central Universities.

\bibliography{}

\vskip 10pt

 {\footnotesize \noindent Xiao-Wu Chen \\
 Key Laboratory of Wu Wen-Tsun Mathematics, Chinese Academy of Sciences\\
School of Mathematical Sciences, University of Science and Technology of China\\
No. 96 Jinzhai Road, Hefei, 230026, Anhui, P.R. China.\\
URL: http://home.ustc.edu.cn/$^\sim$xwchen.}

\end{document}